\documentclass{article}

\usepackage{amsmath, amssymb, amsthm}

\newcommand{\fg}{\mathfrak{g}}
\newcommand{\fr}{\mathfrak{r}}
\newcommand{\fh}{\mathfrak{h}}
\newcommand{\fs}{\mathfrak{s}}
\newcommand{\fa}{\mathfrak{a}}

\newcommand{\N}{\mathbb{N}}
\newcommand{\Z}{\mathbb{Z}}
\newcommand{\R}{\mathbb{R}}
\newcommand{\C}{\mathbb{C}}

\newcommand{\cV}{\mathcal{V}} 
\newcommand{\cN}{\mathcal{N}} 
\newcommand{\cL}{\mathcal{L}}

\newcommand{\SL}{\mathrm{SL}}

\newcommand{\Ad}{\mathrm{Ad}} 
\newcommand{\ad}{\mathrm{ad}} 
\newcommand{\vol}{\mathrm{vol}}

\newtheorem{Theorem}{Theorem}[section]
\newtheorem{Proposition}[Theorem]{Proposition}
\newtheorem{Lemma}[Theorem]{Lemma}
\newtheorem{Corollary}[Theorem]{Corollary}
\theoremstyle{definition}
\newtheorem{Definition}{Definition}[section]
\newtheorem{Remark}{Remark}[section]
\newtheorem{Problem}{Problem}

\begin{document}

\title{Local noncommutative De Leeuw Theorems beyond reductive Lie groups}
\author{
{\it Dedicated to Karl-Hermann Neeb in honour of his 60th birthday}\\[6mm]
B.\ Janssens and B.\ Oudejans
}
\maketitle

\begin{abstract}
Let $\Gamma$ be a discrete subgroup of a unimodular locally compact group~$G$.
In \cite{CJKM23}, it was shown that the $L_p$-norm of a Fourier multiplier $m \colon G \rightarrow \C$ on $\Gamma$ can be bounded locally 
by its $L_p$-norm on $G$, modulo a constant $c(A)$ which depends on the support $A$ of $m|_{\Gamma}$.
In the context where $G$ is a connected Lie group with Lie algebra $\fg$, we develop tools to find explicit bounds on $c(A)$.
We show that the problem reduces to: 
\begin{itemize}
\item[1)] The adjoint representation of the semisimple quotient $\fs = \fg/\fr$ of $\fg$ by the radical $\fr \subseteq \fg$ 
(which was handled in \cite{CJKM23}).
\item[2)] The action of $\fs$ on a set of real irreducible representations that arise from quotients of the commutator series of $\fr$.
\end{itemize}
In particular, we show that $c(G) = 1$ for unimodular connected solvable Lie groups.
\end{abstract}

\section{Introduction}

For a unimodular locally compact group $G$, the left regular representation $L \colon G \rightarrow U(L_2(G))$
defined by $(L(s)\psi)(t) = \psi(s^{-1}t)$ is unitary. For $f \in L_1(G)$ we define $L(f)\psi(s) = \int_{G}f(t)\psi(t^{-1}s)\mu(dt)$ 
by integration against the Haar measure $\mu$. 
The group von Neumann algebra $\cL(G) := L(G)''$ then admits a unique normal semi-finite faithful trace $\tau$, which satisfies
\begin{equation}
	\tau(L(f)^*L(f)) = \langle f, f\rangle_{L_2(G)}
\end{equation}
for $f \in L_1(G)\cap L_2(G)$. This gives rise to the $L_p$-norm $\|x\|_p := \tau(|x|^p)^{1/p}$, and  
the noncommutative $L_p$-space $L_p(\widehat{G})$ is defined as the completion of the space 
$\{x \in \cL(G) \,;\, \|x\|_{p} < \infty\}$ with respect to the norm $\|\,.\,\|_{p}$.

A continuous bounded function $m \colon G \rightarrow \C$ is a \emph{$p$-multiplier} if there exists a bounded linear map 
$T_m \colon L_p(\widehat{G}) \rightarrow L_p(\widehat{G})$ such that $T_m L(f) = L(mf)$ for all $f$ in the twofold convolution product
$C_c(G)*C_c(G)$ of the compactly supported continuous functions \cite{Kunze58, Cooney10, Caspers13}. 
If it exists, it is unique because $\{L(f)\,;\, f \in C_c(G)*C_c(G)\}$
is dense in $L_p(G)$.

To justify the above notations and terminology, note that if $G$ is abelian, then $L_p(\widehat{G})$ coincides with the space of 
$L_p$-functions on the Pontryagin dual $\widehat{G}$ of $G$. The map $T_m$ is then obtained by applying the Fourier transform from 
$G$ to $\widehat{G}$, multiplying with $m$, and then applying the inverse Fourier transform. 

In general, the problem of determining 
whether a given function $m$ is a Fourier multiplier is quite nontrivial even in the commutative case \cite{Stein, GrafakosBook1, GrafakosBook2}. 
In the noncommutative case, recent progress was obtained in \cite{JMPGafa, JungeGonzalezParcet, JMPJEMS, PRS, MeiRicardXu, MeiRicard}, 
see also \cite{CJKM23} for a more extensive introduction to this topic.

In this connection, it is natural to study the relation between the $L_p$-norm \[\|T_m \colon L_p(\widehat{G}) \rightarrow L_p(\widehat{G})\|\]
of a Fourier multiplier $m \colon G \rightarrow \C$ and the $L_p$-norm \[\|T_{m|_{\Gamma}} \colon L_p(\widehat{\Gamma}) \rightarrow L_p(\widehat{\Gamma})\|\]
of its restriction $m_{\Gamma} \colon \Gamma \rightarrow \C$ to a discrete subgroup $\Gamma \subseteq G$.
It is a celebrated theorem of DeLeeuw \cite{DeLeeuw} that these two norms are the same for the inclusion $\Z \subset \R$.
In the noncommutative case, the following result in this direction was obtained in \cite{CJKM23}, relying heavily on on techniques from \cite{CPPR15}.

\begin{Definition}\label{Def:GroupLevel}
Let $G$ be a locally compact group with left Haar measure~$\mu$, 
and let $F \subseteq G$ be a subset. For an open, relatively compact subset $U \subseteq G$, we define
\[
\delta_F(U)  := \frac{\mu( \cap_{g \in F} \;gUg^{-1} )}{ \mu(U)}.
\]
For an open neighbourhood basis $\mathcal{V}$ of the identity of $G$ that consists of symmetric sets $U = U^{-1}$, we define 
$\delta_F(\mathcal{V}) := \liminf_{U \in \mathcal{V}}  \delta_F(U)$. Then $\delta_F$ is defined as the supremum of $\delta_F(\mathcal{V})$ over all open, symmetric neighbourhood bases  $\mathcal{V}$ of the identity.
\end{Definition}

For a subset $A \subseteq G$, we further define
\begin{equation}\label{eq:defc}
c( A ) := \inf \{ \delta_F^{\frac{1}{2}}\,;\, F \subseteq A,  F \: \mathrm{finite}\}.
\end{equation}
We then have the following bound 
on $\Vert T_{m \vert_\Gamma}\colon L_p(\widehat{\Gamma})   \rightarrow L_p(\widehat{\Gamma}) \Vert$
in terms of $c(\mathrm{supp}(m\vert_\Gamma) )$ \cite[Theorem~A]{CJKM23}.

\begin{Theorem}[Local noncommutative De Leeuw Theorem]\label{Thm:ncDeLeeuw}
Let $G$ be a unimodular second countable locally compact group, and let
$m \colon G \rightarrow \C$ be a bounded, continuous function. Then for every $1 \leq  p < \infty$ we have
\begin{equation}\label{eq:Lpestimate}
c(\mathrm{supp}(m\vert_\Gamma) )
\cdot \Vert T_{m \vert_\Gamma}\colon L_p(\widehat{\Gamma})   \rightarrow L_p(\widehat{\Gamma}) \Vert \leq
\Vert T_{m  }\colon L_p(\widehat{G}) \rightarrow L_p(\widehat{G}) \Vert.
\end{equation}
\end{Theorem}

The strength of this result is that the constants $\delta_{F}$, and hence the constants $c(A)$, can be determined explicitly in many interesting situations.
For real reductive Lie groups $G$ (which are automatically unimodular), the following explicit bound on $\delta_{F}$ was obtained in \cite[Theorem~B]{CJKM23}.
It results in nontrivial bounds for \eqref{eq:Lpestimate} in the \emph{local} case, where $m\colon G \rightarrow \C$
is compactly supported.

Let $\fg$ be the Lie algebra of $G$, let $\mathrm{Ad} \colon G \rightarrow \mathrm{GL}(\fg)$ be the adjoint representation, 
and let $B_{\theta}(x,y) = -B(x, \theta y)$ be the inner product derived from the invariant bilinear form $B$ and the Cartan involution $\theta$.

\begin{Theorem}[Reductive groups]\label{Thm:boundReductiveGroup}
Let G be a real reductive Lie group, and let $F \subseteq G$ be a relatively compact, symmetric subset that contains the identity.
Then 
\begin{equation} 
\delta_{F}  \geq \rho^{-d/2}, 
\end{equation}
where $\rho := \sup_{g\in F}\|\Ad_g\|$ is the maximal norm of $F$ in the adjoint representation with respect to the inner product $B_{\theta}$, and 
$d$ is the maximal dimension of a nilpotent adjoint orbit.
\end{Theorem}

The aim of the present paper is to develop methods for bounding $\delta_F$ for connected Lie groups $G$ beyond the 
class of reductive Lie groups.
In view of Theorem~\ref{Thm:ncDeLeeuw}, this immediately translates to bounds on the $L_p$ norm 
of a Fourier multiplier $m|_{\Gamma}$ for discrete subgroups $\Gamma \subseteq G$.
We show that $c(G) = 1$ for unimodular connected solvable Lie groups
(Corollary~\ref{Thm:Solvable}). More generally, 
 for an arbitrary connected Lie group $G$, we decompose $\delta_{F}$ into a bound 
that depends only on the adjoint representation of the Levi subalgebra $\fs \subseteq \fr$ (where Theorem~\ref{Thm:boundReductiveGroup} applies), 
and bounds that arise from the decomposition of the radical $\fr \subseteq \fg$ into irreducible representations for the Levi factor $\fs$.
We describe the remaining open problems, and indicate a direction of further research.

\paragraph{Acknowledgements} We would like to thank the anonymous referee for the careful reading, and for valuable suggestions.

\section{Linearization}

Let $G$ be a Lie group, and let $(\pi, V)$ be a continuous representation of $G$ on a finite-dimensional real vector space $V$. We denote the Lebesgue measure 
on $V$ by $\Lambda$. 
A subset $F\subseteq G$ is called \emph{symmetric} if $F = F^{-1}$, and $U\subseteq V$ is called symmetric if $U = -U$.
The following is a linearized version of Definition~\ref{Def:GroupLevel}.
\begin{Definition}\label{def:LAlevel}
For a relatively compact subset $F \subseteq G$ and a relatively compact open neighbourhood $U \subseteq V$ of $0\in V$, we define 
\begin{equation}
\delta^{V}_{F}(U) := \frac{\Lambda\big(\cap_{g\in F}\pi(g)(U)\big)}{\Lambda(U)}.
\end{equation}
For a neighbourhood basis $\cV$ of $0\in V$ that consists of symmetric open sets, we set
$
\delta^{V}_{F}(\cV) := \lim\inf_{U\in \cV}\delta^V_{F}(U),
$
and we define $\delta^V_{F}$ to be the supremum of $\delta^{V}_{F}(\cV)$ over all such neighbourhood bases.
\end{Definition}

According to the following minor generalization of \cite[Prop.~8.5]{CJKM23}, 
it suffices to consider the above linearized version of $\delta_{F}$ for the adjoint representation.

\begin{Proposition}[Linearization]
For a Lie group $G$, we have $\delta_{F} = \delta_F^{\fg}$ for the adjoint representation $\fg$ of $G$.
\end{Proposition}
\begin{proof}
Since the exponential map $\exp \colon \fg \rightarrow G$ is a local diffeomorphism around zero, we may assume 
without loss of generality that the sets $U_G \subseteq G$ in the neighbourhood basis $\cV_{G}$ for $1\in G$ correspond
to sets $U_{\fg} \subseteq \fg$ in a neighbourhood basis $\cV_{\fg}$ for $0\in \fg$ by 
$U_{G} = \exp(U_{\fg})$, and that the restriction $\exp \colon U_{\fg} \rightarrow U_{G}$ is a diffeomorphism.
Since the exponential map is $G$-equivariant (so $\exp \circ \mathrm{Ad}_{g}(x) = g \exp(x)g^{-1}$ for all $g\in G$ and $x\in \fg$), 
we have 
\[
\frac{\mu(\cap_{g\in F} \;g U_{G} g^{-1})}{\mu(U_{G})} = \frac{\mu\circ \exp (\cap_{g\in F} \Ad_g(U_{\fg}))}{\mu\circ \exp (U_{\fg})},
\]
so it remains to compare $\mu \circ \exp$ to the Lebesgue measure on small zero-neighbourhoods.

Let $\vol_{G}$ be a left invariant volume form on $G$ such that $\mu(U) = \int_{U}\vol_{G}$ for open relatively compact subsets $U\subseteq G$, 
and let $\vol_{\fg}$ be a translation invariant volume form on $\fg$ such that $\exp^*\vol_{G}|_{0} = \vol_{\fg}|_{0}$ at the origin.
Then $\exp^*\vol_G = \nu \vol_{\fg}$ for a smooth function $\nu \colon \fg \rightarrow \R$ with $\nu(0) =1$. 
(In fact, $\nu(x) = \mathrm{det}(D^{L}_{x}\exp)$ is the determinant of the left logarithmic derivative 
\[
D^L\exp_{x} = \frac{\mathrm{Id}_{\fg} - \exp(-\ad_{x})}{\ad_x}
\]
of the exponential map \cite[Prop.~8.5]{CJKM23}, but we will not need this fact here.)
Since $\nu$ is smooth on $\fg$, there exists a radius $R>0$ and a
constant $C>0$ such that $|\nu(x) - 1| \leq C \|x\|$ for all $x$ in the ball $B_{R}(0)$ of radius $R$
around zero.
If $U_{\fg}$ is contained in $B_{\varepsilon}(0)$ with $0<\varepsilon<R$, we then have 
$(1-\varepsilon) \Lambda(W) < \mu \circ \exp(W) \leq (1 + \varepsilon) \Lambda(W)$ for all 
open $W \subseteq U_{\fg}$. It follows that
\[
\left(
\frac{1-\varepsilon}{1+\varepsilon}
\right)
\frac{\Lambda (\cap_{g\in F} \Ad_g(U_{\fg}))}{\Lambda (U_{\fg})}
\leq
\frac{\mu\circ \exp (\cap_{g\in F} \Ad_g(U_{\fg}))}{\mu\circ \exp (U_{\fg})}
 \leq 
 \left(
 \frac{1+\varepsilon}{1-\varepsilon}
 \right)
\frac{\Lambda (\cap_{g\in F} \Ad_g(U_{\fg}))}{\Lambda (U_{\fg})}.
\]
Since any ball $B_{\varepsilon}(0)$ contains an open set $U_{\fg}$ from $\cV$, the result follows.
\end{proof}

\section{Trivial bounds and the Reduction Lemma}

After proving a trivial bound on $\delta^{V}_{F}$, 
we prove a Reduction Lemma that bounds $\delta^{V}_{F}$ in terms of $\delta^{W}_{F}$ and $\delta^{V/W}_{F}$
for every subrepresentation $W\subseteq V$.
If $G$ is a connected Lie group with Lie algebra $\fg$, then every ideal $\fh \subseteq \fg$ is a subrepresentation 
of the adjoint representation. In particular, the Reduction Lemma can be used to bound $\delta^{\fg}_{F}$ 
in terms of $\delta^{\fh}_{F}$ and $\delta^{\fg/\fh}_{F}$.


\subsection{The trivial bound}\label{sec:trivialbound}
We equip $V$ with an inner product, and define 
\begin{equation}
\rho := \sup_{g\in F} \|\pi(g)\|.
\end{equation} 
We then obtain the following trivial bound on $\delta_{F}^{V}$. 
\begin{Proposition}[Trivial bound]\label{Prop:trivialbound}
For every relatively compact, symmetric subset $F\subseteq G$, we have
\begin{equation}\label{eq:TrivialBound}
\delta_{F}^{V} \geq \rho^{-\mathrm{dim}(V)}.
\end{equation}
\end{Proposition}
\begin{proof}
Let $\cV$ be the neighbourhood basis consisting of open balls $B_{r}(0)$ of radius~$r$.
Since $\pi(g)(B_{r}(0)) \subseteq B_{\rho r}(0)$, we have 
$B_{r}(0) \subseteq \pi(g^{-1})B_{\rho r}(0)$ for all $g\in F$. Since $F$ is symmetric ($F = F^{-1}$), 
we have $B_r(0) \subseteq \bigcap_{g\in F} \pi(g)B_{\rho r}(0)$, and 
\[
\frac{\Lambda\Big(\bigcap_{g\in F} \pi(g)B_{r}(0)\Big)}{\Lambda(B_{r}(0))} \geq \frac{\Lambda(B_{r/\rho}(0))}{\Lambda(B_{r}(0))} = \rho^{-\mathrm{dim}(V)}
\]
as required.
\end{proof}
\begin{Remark}
Note that unlike the quantity $\delta_{F}^{V}$, 
the bound \eqref{eq:TrivialBound} depends on the choice of inner product.
\end{Remark}
If $A$ is a symmetric, relatively compact subset of a Lie group $G$, then 
the constant $c(A)$ from \eqref{eq:defc} satisfies $c(A) \geq \rho^{-\frac{1}{2}\mathrm{dim}(\fg)}$ 
for $\rho = \sup_{g\in A}\|\mathrm{Ad}_{g}\|$. In particular, this shows that if $m$ is compactly supported, the bound \eqref{eq:Lpestimate} 
is not vacuous.
\subsection{The Reduction Lemma}\label{sec:reductionlemma}

If a subrepresentation $W\subseteq V$ admits a complementary $G$-representation (so that $V$ decomposes as 
$V = W \oplus W'$), then it is not hard 
to see that 
\[\delta^{V}_{F} \geq \delta^{W}_{F} \delta_{F}^{V/W}.\]
 By the following lemma, this remains true 
even if $W$ is \emph{not} complemented as a $G$-representation.

\begin{Lemma}[Reduction Lemma]\label{Lemma:BenjaminsTruc}
Let $F\subseteq G$ be a relatively compact, symmetric subset. Let $W \subseteq V$ be a subrepresentation, and let $\cV_{V/W}$ and $\cV_{W}$
be open neighbourhood bases of the origin in $V/W$ and $W$, respectively. 
Then there exists an open neighbourhood basis $\cV$ of the origin in $V$ such that 
\begin{equation}\label{eq:BenjaminstrucVergelijking}
\delta^{V}_{F}(\cV) \geq \delta^{V/W}_{F}(\cV_{V/W})\delta^{W}_{F}(\cV_{W}),
\end{equation}
and such that $\cV_{V}$ is symmetric if $\cV_{V/W}$ and $\cV_{W}$ are symmetric.
\end{Lemma}
We introduce some notation that will be useful in the proof.
For a real representation $(\pi,V_{\pi})$ and an open subset $A\subseteq V_{\pi}$, we set $A^F := \bigcap_{g\in F} \big(\pi(g)(A)\big)$.
Further, we choose an inner product on $V_{\pi}$ and set  
\[A^{r} := \bigcap_{x \in B_{r}(0)} (A + x),\]
where $B_r(0)$ is the ball of radius $r$ around $0$.
Note that $A = \bigcup_{r>0} A^r$ since $A$ is open, so we find
 \begin{equation}\label{eq:Fatou}
 \lim_{r\downarrow 0}\Lambda(A^r) = \Lambda(A).
 \end{equation} 
 More generally, for relatively compact subsets $F \subseteq G$ we have:
\begin{Proposition}\label{Prop:GroepEnSchuif}
$\lim_{r\downarrow 0} \Lambda((A^r)^F) = \Lambda(A^F)$.
\end{Proposition}
\begin{proof}
Since 
$\pi(g)B_r(0) \subseteq B_{\rho r}(0)$ for all $g\in F$,   
we have
\begin{eqnarray*}
(A^r)^F &=& \bigcap_{g\in F}\bigcap_{x\in B_r(0)} \pi(g)(A + x)\\
&=& \bigcap_{x\in B_r(0) \text{ and } g\in F} \pi(g)(A) + \pi(g)x\\
&\supseteq& \bigcap_{y\in B_{\rho r}(0)}\bigcap_{g\in F}  \pi(g)(A) + y  =  (A^F)^{\rho r}.
\end{eqnarray*}
Since $(A^F)^{\rho r} \subseteq (A^{r})^F \subseteq A^{F}$, equation~\eqref{eq:Fatou} yields $\lim_{r\downarrow 0} \Lambda((A^r)^F) = \Lambda(A^F)$.
\end{proof}

\begin{proof}[Proof of Lemma~\ref{Lemma:BenjaminsTruc}]
Fix an inner product on $V$, and identify $W^{\perp}$ with $V/W$ by means of the quotient map. 
For $U_W \in \cV_W$, $U_{V/W} \in \cV_{V/W}$ and $\varepsilon > 0$, 
we set
\[
U_{\varepsilon} := U_W \times \varepsilon U_{V/W}  \subseteq W \times V/W \simeq V.
\]
Let $P_W$ and $P_{W^{\perp}}$ be the orthogonal projections onto $W$ and $W^{\perp}$, respectively.
For $w \in W$ and $w' \in W^{\perp}$, we then have
\begin{equation}\label{eq:schuif}
\pi(g)(w \oplus \varepsilon w') = (\pi_W(g) w + \varepsilon P_{W}\pi(g)w')\oplus \varepsilon P_{W^{\perp}}\pi(g)w'
\end{equation}
in $W \oplus W^{\perp}$.

Without loss of generality, we may assume that $U_{V/W}$ is contained in a ball of radius $1$. 
Let $r := \varepsilon\rho$, where $\rho := \sup_{g\in F}\|\pi(g)\|$ as before.
 For $w \in ((U_W)^{r})^{F}$ and $w' \in U_{V/W}^{F}$, we then have
$\pi(g) w \in (U_{W})^r$ and $\|\varepsilon P_{W}\pi(g)w'\| \leq \varepsilon \rho$, 
so the first component of \eqref{eq:schuif} is contained in $U_{W}$.
Since $[P_{W^{\perp}}\pi(g)w'] = \pi_{V/W}(g)([w'])$ in $V/W$, the second component of \eqref{eq:schuif}
is contained in $\varepsilon U_{V/W}$.
It follows that \[\pi(g) \Big(((U_W)^{r})^{F} \times \varepsilon U_{V/W}^{F}\Big) \subseteq U_W \times \varepsilon U_{V/W}\]
for all $g\in F$. Since $F$ is symmetric, we conclude that 
\[
((U_W)^{r})^{F} \times \varepsilon U_{V/W}^{F} \subseteq \bigcap_{g\in F} \pi(g)(U_W \times \varepsilon U_{V/W}) = U_{\varepsilon}^{F}.
\]
So 
\[
\frac{\Lambda(U_{\varepsilon}^{F})}{\Lambda(U_{\varepsilon})} \geq 
\frac{\varepsilon \Lambda\Big(\big((U_{W})^r\big)^{F}\Big) \Lambda(U_{V/W}^F)}{\varepsilon\Lambda(U_W)\Lambda(U_{V/W})},
\]
and since $\lim_{r\downarrow 0} \Lambda((U_{W})^r)^{F} = \Lambda(U_{W}^{F})$ by Proposition~\ref{Prop:GroepEnSchuif}, we conclude that 
\begin{equation}
\liminf_{\varepsilon \downarrow 0} \delta_{F}^{V}(U_{\varepsilon}) \geq \delta_{F}^{W}(U_W)\delta_{F}^{V/W}(U_{V/W}).
\end{equation}
The neighbourhood basis $\cV_{V}$ is constructed as follows. Select for every $n\in \N$ an open set $U^{n}_{W} \in \cV_W$ 
that is contained in a ball of radius $1/n$, and such that $U^{n}_{W}\subseteq U^{n-1}_{W}$ for $n> 1$.
Select $U^n_{V/W}$ in a similar fashion, and select $0< \varepsilon_{n} \leq 1$ such that $\varepsilon_{n} \leq \varepsilon_{n-1}$ for $n\in \N$, 
and  \[\delta_{F}^{V}(U_{\varepsilon_{n}}) \geq {(1-{\textstyle \frac{1}{n}})}\delta_{F}^{W}(U^n_W)\delta_{F}^{V/W}(U^n_{V/W}).\]
Then $\cV_{V} = \{U^n_{W} \times \varepsilon_{n}U^n_{V/W}\,;\, n\in \N\}$
is a symmetric neighbourhood basis of ${0\in V}$ that satisfies \eqref{eq:BenjaminstrucVergelijking}.
\end{proof}

\subsection{Connected Lie groups}

Let $G$ be a connected Lie group with Lie algebra $\fg$.
Let 
$\fr \subseteq \fg$ be the radical of $\fg$, i.e.\ the maximal solvable ideal. 
Then $\fr$ integrates to a closed, connected, normal Lie subgroup $R \subseteq G$, 
and the semisimple quotient $S := G/R$ has Lie algebra $\fs := \fg/\fr$ \cite[Thm.~3.18.13]{Varadarajan}.
Let $F \subseteq G$ be a relatively compact, symmetric subset.
Using the Reduction Lemma~\ref{Lemma:BenjaminsTruc}, we obtain a bound 
\begin{equation} \delta^{\fg}_{F} \geq \delta^{\fs}_{[F]}\delta^{\fr}_{F} \end{equation} 
with two types of contributions.
Since $R\subseteq G$ acts trivially on $\fs = \fg/\fr$, the adjoint action of $G$ on $\fs$ factors through the adjoint representation of $S = G/R$.
The first factor $\delta^{\fs}_{[F]}$ therefore depends only on the adjoint representation of $S$, and on the image 
$[F] \subseteq S$ of $F\subseteq G$ under the quotient map $G \rightarrow S$. Since $S$ is semisimple, 
the factor $\delta^{\fs}_{[F]}$ can be bounded using Theorem~\ref{Thm:boundReductiveGroup}.

In the remainder of this paper we therefore focus on the second factor $\delta^{\fr}_{F}$, which 
depends only on the adjoint action of $G$ on the subrepresentation $\fr \subseteq \fg$.

\section{Solvable Lie groups and Lie algebras}

Let $G$ be a connected Lie group that acts by automorphisms on a solvable Lie algebra $\fr$, and let $F\subseteq G$ be a relatively compact, 
symmetric subset.
We use the structure theory of solvable Lie algebras to further decompose $\delta_{F}^{\fr}$ into factors that depend
only on the semisimple quotient $S = G/R$. 
The main tool for this is Lemma~\ref{Lemma:shift}, an extension of the Reduction Lemma~\ref{Lemma:BenjaminsTruc} that allows us to shift characters from 
a subrepresentation $W \subseteq V$ to the quotient $V/W$. We apply this to the $G$-representation $\fr$, whose subrepresentations 
are essentially given by Lie's Theorem.
 

\subsection{Lie's Theorem for real solvable Lie algebras}

In order to set notation, we briefly review Lie's Theorem for the adjoint representation of a solvable Lie algebra. 

\subsubsection{Root spaces}
For a Lie algebra $\fr$, the commutator series is defined inductively by $\fr^0 := \fr$, and $\fr^{k+1} := [\fr^k, \fr^k]$ for $k>0$.
We say that $\fr$ is $r$-step solvable if $\fr^{r} = \{0\}$, in which case we have a finite filtration 
\begin{equation}\label{eq:filter}
\fr = \fr^{0} \supset \fr^{1} \supset \ldots \supset \fr^{r-1} \supset \fr^{r} = \{0\}.
\end{equation}
Since $\fr^k$ is an ideal in $\fr$ \cite[Thm.~3.7.1]{Varadarajan}, we obtain a representation $(\pi^k, V^k)$ of $\fr$ by setting
$V^k := \fr^{k}/\fr^{k+1}$, 
and defining $\pi^k \colon \fr \rightarrow \mathrm{End}(V^k)$ by
\begin{equation}
\pi^k(x)([v]) := [\mathrm{ad}_{x}(v)] 
\end{equation} 
for $x\in \fr$ and $[v] \in V^k$. Note that $\fr^k$ acts trivially on $\fr^{k}/\fr^{k+1}$, so the representation $\pi^k \colon \fr \rightarrow \mathrm{End}(V^k)$ factors through the quotient $\fr/\fr^{k}$.
In particular, $(\pi^0, V^0)$ is the trivial representation.

Let $(\pi, V)$ be any representation of $\fr$, and let $\lambda \in \fr^*$. If 
\begin{equation}
V_{\lambda}:= \{v \in V \,;\, \pi(x)v = \lambda(x) v \text{ for all } x\in \fr\}
\end{equation}
is nonzero, then $\lambda$ is called a \emph{weight} of $V$, and $V_{\lambda}$ is called a \emph{weight space}.
Since every weight $\lambda \colon \fr \rightarrow \C$ is a Lie algebra homomorphism into an abelian Lie algebra, it vanishes 
on $[\fr,\fr]$, and we can identify the set of weights with a subset $W(V)$ of $(\fr/[\fr,\fr])^*$.
The weights of $(V^k, \pi^k)$ are called the \emph{roots of order $k$}. We denote the set of $k^{\mathrm{th}}$
order roots by $\Phi^k(\fr) := W(V^k)$, and we define 
\begin{equation}
\Phi(\fr) := \{(k, \lambda) \,;\, k \in \{0, \ldots, r-1\}, \lambda \in \Phi^{k}(\fr)\}
\end{equation}
to be the set of all roots, labelled by $k$. (This is because the same root $\lambda$ can occur in different representations $V^k \neq V^{k'}$.)

%

\subsubsection{Automorphisms}
 
 Let $G$ be a (not necessarily connected) Lie group with Lie algebra $\fg$, let $\fr$ be a solvable Lie algebra, 
and let $\alpha \colon G \times \fr \rightarrow \fr$ 
be a continuous action of $G$ on $\fr$ by automorphisms. 
We write $\alpha_g(x)$ for the action of $g\in G$ on $x\in \fr$. 
The automorphisms $\alpha_g$ of $\fr$ restrict to automorphisms on $\fr^k$ and $\fr^{k+1}$, so that $V^k = \fr^k/\fr^{k+1}$
is a $G$-representation, which we denote by $\pi_{G}^{k} \colon G \rightarrow \mathrm{GL}(V^{k})$. 
If $\lambda \in \Phi^k(\fr)$ is a root of order $k$, and $V^k_{\lambda}$ is its weight space,
then $\lambda \circ \alpha_{g^{-1}}$ 
is a root of order $k$ as well, and $\pi_{G}^k(g)$ maps $V^k_{\lambda}$ to $V^k_{\lambda \circ \alpha_{g^{-1}}}$.
Indeed, for $[v] \in V^k_{\lambda}$ and $x\in \fr$, we have  
\begin{eqnarray*}
\ad_{x} \big( \pi^k_{G}(g) ([v])\big) &=&
\mathrm{ad}_{x}([\alpha_g(v)]) = [\alpha_g([\alpha_{g^{-1}}(x), v])]\\
&=& \lambda\big(\alpha_{g^{-1}}(x)\big) [\alpha_g v]
= \lambda \circ \alpha_{g^{-1}}(x) \big(\pi^k_{G}(g) ([v])\big).
\end{eqnarray*}

Since the action of $G$ on the discrete set $\Phi^k(\fr) \subseteq (\fr/[\fr,\fr])^*$ of roots is continuous, the connected identity component $G_0$
fixes the weights, and every $V^{k}_{\lambda}$ is a $G_0$-subrepresentation, which we denote by 
$\pi_{\lambda}^{k} \colon G_0 \rightarrow \mathrm{GL}(V^{k}_{\lambda})$.
 
\subsubsection{Lie's Theorem for the adjoint representation}\label{sec:LieThm}
 
Let $\fr$ be a solvable real Lie algebra, and let $\fr_\C := \fr \otimes_{\R}\C$ be its complexification.
\begin{Theorem}[Lie's Theorem (Thm.~3.7.3 in \cite{Varadarajan})]
Every finite dimensional complex representation $(V,\pi)$ of $\fr_{\C}$ has a weight space.
\end{Theorem}

We will need the following slight reformulation of Lie's Theorem in the particular context of the adjoint representation.
Although it follows directly from Lie's Theorem, it is instructive to give a direct proof. 
\begin{Theorem}[Lie's Theorem, special case]
The $\fr_{\C}$-representation $V^k = \fr^k_{\C}/\fr^{k+1}_{\C}$ decomposes as a direct sum of weight spaces,
\begin{equation}
V^k = \bigoplus_{\lambda \in \Phi^k(\fr_{\C})} V^{k}_{\lambda}.
\end{equation}
\end{Theorem}
\begin{proof}
Suppose that $\fr_{\C}$ is $r$-step solvable. If $r = 1$, then the statement holds because the action of $\fr_{\C}$ on $V^0$ is trivial.
Suppose by induction that the statement holds for all Lie algebras that are $(r-1)$-step solvable. 
Then $V^0$ is trivial, and
for $k\geq 1$, the space $V^k$ decomposes into weight spaces 
\begin{equation}\label{eq:WeightsFroR1}
V^k=  \bigoplus_{\mu \in \Phi^{k-1}(\fr_{\C}^1)} V_{\mu}^{k-1}
\end{equation}
as an $\fr_{\C}^1$-representation, because $V^k(\fr_{\C}) = \fr_{\C}^k/\fr_{\C}^{k+1}$ is identical to $V^{k-1}(\fr_{\C}^1) = (\fr_{\C}^1)^{k-1}/(\fr_{\C}^{1})^k$,
and $\fr_{\C}^1$ is $(r-1)$-step solvable. 
Let $R \subseteq \mathrm{Aut}(\fr)$ be 
the closed subgroup of $\mathrm{GL}(\fr)$ generated by $\exp(\fr)$.
Since $R$ is connected and acts on $\fr_{\C}^1$ by automorphisms, it respects the weight space
decomposition \eqref{eq:WeightsFroR1}. 
In particular, the weight spaces $V^{k-1}_{\mu}$ 
are subrepresentations 
of the $\fr_{\C}$-representation $(\pi^k, V^k)$.
Since the commutator 
\[[\pi^k(x)|_{V^{k-1}_{\mu}}, \pi^k(y)|_{V^{k-1}_{\mu}}] = \mu([x,y]) \mathrm{Id}_{V^{k-1}_{\mu}}\] has 
zero trace for all $x,y \in \fr$, we conclude that $\mu \in (\fr_{\C}^1)^*$ is identically zero. So $V^k(\fr_{\C}) = V^{k-1}(\fr^1_{\C})$ is trivial as an $\fr_{\C}^1$-representation, 
and  the $\fr_{\C}$-representation $(\pi^k, V^k)$ factors through $\fr_{\C}/\fr_{\C}^1$. Because this is an abelian Lie algebra, $V^k$
decomposes into root spaces $V^k_{\lambda}$ for a set of $k^{\mathrm{th}}$ order roots $\lambda \in (\fr_{\C}/[\fr_{\C}, \fr_{\C}])^*$.
\end{proof}

Since $\fr_{\C}$ is the complexification of the real Lie algebra $\fr$, the complex conjugation $\overline{\lambda}(x) := \overline{\lambda(\overline{x})}$
permutes the $k^{\mathrm{th}}$ order roots, and we have $\overline{V}{}_{\lambda}^{k} = V^{k}_{\overline{\lambda}}$ 
for the weight spaces of the $\fr_{\C}$-representation $V^{k}(\fr_{\C}) = \fr^k_{\C}/\fr^{k+1}_{\C}$.
It follows that the real $\fr$-representation $V^k(\fr) = \fr^k/\fr^{k+1}$ does not necessarily decompose into weight spaces. 
Rather, the isotypical components
are 
\begin{eqnarray}
U_{\lambda}^{k} &:=& \big(V^{k}_{\lambda} \oplus V_{\overline{\lambda}}^{k}\big) \cap (\fr^{k}/\fr^{k+1}) \quad\text{ if } \lambda \neq \overline{\lambda},\\
U_{\lambda}^{k} &:=& V_{\lambda}^{k}\cap (\fr^{k}/\fr^{k+1}) \qquad\text{ if } \lambda = \overline{\lambda}.
\end{eqnarray}
The real $\fr$-representation $V^k(\fr) = \fr^k/\fr^{k+1}$ then decomposes as
\begin{equation}\label{eq:decompositionrG}
V^k(\fr) = \bigoplus_{\{\lambda, \overline{\lambda}\}\subseteq \Phi^k(\fr_\C)} U^k_{\lambda}.
\end{equation}

\subsection{Shifting characters}
In order to take full advantage of the decomposition~\eqref{eq:decompositionrG}, we will need a generalization of the 
Reduction Lemma~\ref{Lemma:BenjaminsTruc} that 
allows us to shift characters from a subrepresentation $W\subseteq V$ to the quotient $V/W$.

\subsubsection{Motivating example}
The motivating example is the representation $(\pi, \R^2)$ of $G = (\R^{>0}, \,\cdot\,)$ defined by
\[
	\pi(a) = \begin{pmatrix}a & 0 \\ 0 & a^{-1}\end{pmatrix}.
\]
In this case, the bound 
\[\delta^{\R^2}_{F} \geq \rho^{-2}\] obtained from the 
Reduction Lemma is clearly suboptimal. 
Indeed, intersecting the invariant neighbourhood $\pi(G)B_{\varepsilon}(0) = \{(x,y)\in \R^2\,;\, {|xy| < \frac{1}{2}\varepsilon^2}\}$
with the relatively compact square $C_{R} = \{(x,y) \in \R^2\,;\, |x| < R, |y| < R \}$ 
with sides $R \gg \varepsilon$, one obtains open neighbourhoods $U_{R,\varepsilon} = C_R \cap (\pi(G)B_{\varepsilon}(0))$
with volume 
\[
\Lambda(U_{R,\varepsilon}) = 2\varepsilon^2\left(1+\log\Big(\frac{2R^2}{\varepsilon^2}\Big)\right)
\]
that satisfy $\bigcap_{g\in F} \pi(g)U_{R,\varepsilon} \supseteq U_{R/\rho, \varepsilon}$. For $R = \sqrt{\varepsilon}$, the neighbourhood basis 
consisting of $U_{\sqrt{\varepsilon}, \varepsilon}$ then yields the sharp bound
\[
\delta^{\R^2}_{F} \geq \lim_{\varepsilon\downarrow 0} \frac{\Lambda(U_{\rho^{-1}\sqrt{\varepsilon}, \varepsilon})}{\Lambda(U_{\sqrt{\varepsilon}, \varepsilon})} = 
\lim_{\varepsilon\downarrow 0} \frac{2\varepsilon^2(1+\log(\frac{2\varepsilon}{\rho^2\varepsilon^2}))}{2\varepsilon^2(1+\log(\frac{2\varepsilon}{\varepsilon^2}))}= 1.
\]

\subsubsection{Neighbourhoods of logarithmic volume}
Now let $V_1$ and $V_2$ be real vector spaces of dimension $d_1$ and $d_2$, respectively, and let 
$U\subseteq V_1$ and $V \subseteq V_2$ be relatively compact open subsets.
Mimicking the above construction, we define $U\times_{R_1,R_2,\varepsilon} V \subseteq V_1\times V_2$ for $\varepsilon \leq R_1^{d_1}R_2^{d_2}$ by
\begin{equation}\label{eq:logproduct}
 U\times_{R_1,R_2,\varepsilon} V := \{(\alpha u, \beta v)\,;\, u\in U, v \in V, |\alpha|\leq R_1, |\beta| \leq R_2, |\alpha^{d_1}\beta^{d_2}|\leq \varepsilon \}.
\end{equation}
These sets are almost invariant under the $\R^{>0}$-action on $U\times V$ defined by $\mu \cdot (u,v) := (\mu^{\frac{1}{d_1}}u, \mu^{-\frac{1}{d_2}}v)$ in the following sense. For a symmetric set $F \subseteq \R^{>0}$, we set $D := \sup_{\mu \in F}|\mu|$.
Then $\mu \cdot (U\times_{R_1,R_2,\varepsilon} V) \subseteq U\times_{D^{1/d_1}R_1,D^{1/d_2}R_2,\varepsilon} V$ for all $\mu \in F \subseteq \R^{+}$, with the same $\varepsilon$ on both sides.

\begin{Definition}
A subset $U$ of a real vector space is \emph{balanced} if $\alpha U \subseteq U$ for all real numbers $\alpha$ with $|\alpha|\leq 1$.
\end{Definition}

\begin{Lemma}\label{lem:balanceboundary}
If $U$ and $V$ are balanced, then the inequality 
$|\alpha^{d_1}\beta^{d_2}| \leq \varepsilon$ in \eqref{eq:logproduct} can be replaced by  $|\alpha^{d_1}\beta^{d_2}| = \varepsilon$.
\end{Lemma}
\begin{proof}
If $\alpha^{d_1} \leq \varepsilon/R_2^{d_2}$, then $(\alpha u, \beta v) = (\widetilde{\alpha}\widetilde{u}, \widetilde{\beta}\widetilde{v})$ with 
$\widetilde{\alpha} = (\varepsilon/R_2^{d_2})^{1/d_1}$, $\widetilde{\beta} = R_2$,  $\widetilde{u} = (\alpha/\widetilde{\alpha})u$ and $\widetilde{v} = (\beta/\widetilde{\beta})v$. Then $\widetilde{u} \in U$ and $\widetilde{v} \in V$ because $U$ and $V$ are balanced, $|\widetilde{\alpha}| \leq R_1$ because 
$\varepsilon \leq R_1^{d_1}R_2^{d_2}$, and $|\widetilde{\alpha}{}^{d_1} \widetilde{\beta}{}^{d_2}| =\varepsilon$ by definition.

If $\alpha^{d_1} \geq \varepsilon/R_2^{d_2}$, then $(\alpha u, \beta v) = (\alpha u, \widetilde{\beta}\widetilde{v})$ for
$\widetilde{\beta} = (\varepsilon/\alpha^{d_1})^{1/d_2}$ and $\widetilde{v} = (\beta/\widetilde{\beta})v$. Then $\widetilde{v} \in V$ because 
$V$ is balanced and $(\beta/\widetilde{\beta})^{d_2} =  \alpha^{d_1}\beta^{d_2}/\varepsilon \leq 1$, and $|\widetilde{\beta}|\leq R_2$
because $\alpha^{d_1} \geq \varepsilon/R_2^{d_2}$.
\end{proof}

\begin{Lemma}\label{lem:logaritmetruc}
If $U$ and $V$ are balanced and $0 < \varepsilon < R_1^{d_1}R_2^{d_2}$, then
\[\Lambda(U\times_{R_1,R_2,\varepsilon} V) =\varepsilon\left(1 + \log\Big(\frac{R_1^{d_1}R_2^{d_2}}{\varepsilon}\Big)\right) \Lambda(U)\Lambda(V).\]
\end{Lemma}
\begin{proof}
Scaling $V_1$ by $R_1$ and $V_2$ by $R_2$, we find
\begin{equation}\label{eq:rescale}
\Lambda(U\times_{R_1,R_2,\varepsilon} V) = R_1^{d_1}R_2^{d_2}v(\varepsilon/R_1^{d_1}R_2^{d_2}),
\end{equation}
where $v(\varepsilon):= \Lambda(U\times_{1,1,\varepsilon} V)$.
For an ascending sequence $\varepsilon = \alpha^{d_1}_0 < \ldots < \alpha^{d_1}_N = 1$ and the corresponding descending sequence $\beta^{d_2}_i = \varepsilon/\alpha^{d_1}_i$, we find
\[
 A_N \subseteq U\times_{1,1,\varepsilon} V \subseteq B_N,
\]
where $A_N = \bigcup_{n=0}^{N-1}(\alpha_n U) \times (\beta_{n+1}V)$ and
$B_N = \bigcup_{n=0}^{N-1}(\alpha_{n+1} U) \times (\beta_{n}V)$. (The second inclusion uses that both $U$ and $V$ are balanced).
To find an underestimate for $v(\varepsilon)$, write $A_N$ as the disjoint union
\[
 A_N = (\alpha_0 U)\times (\beta_1 V) \sqcup \bigsqcup_{n=0}^{N-1}(\alpha_{n+1} U \setminus \alpha_{n} U) \times (\beta_{n+1}V).
\]
Since $U$ is balanced, $\alpha_{n} U \subseteq \alpha_{n+1} U$, so that
\[\Lambda(A_N) = \Lambda(U)\Lambda(V)\Big(\alpha_0^{d_1}\beta^{d_2}_1  + \sum_{n=0}^{N-1}\beta^{d_2}_{n+1}(\alpha^{d_1}_{n+1} - \alpha^{d_1}_{n})\Big).
\] Setting $x_n = \alpha^{d_1}_n$ and $y(x) = \frac{\varepsilon}{x}$, one sees that as the subdivision becomes finer, this converges to $\Lambda(U)\Lambda(V)(\varepsilon + \int_{\varepsilon}^{1}\frac{\varepsilon}{x}) dx$. Since the upper bound from $B_N$ gives the same result,  we have $v(\varepsilon) = \varepsilon(1 + \log(\frac{1}{\varepsilon}))\Lambda(U)\Lambda(V)$,
which combined with \eqref{eq:rescale} yields the desired result.
\end{proof}

\subsubsection{Rescaling representations by a character}
Let $G$ be a Lie group, and $(\pi,V)$ a real representation of dimension $d$.
Let $\chi \colon G \rightarrow \R^{+}$ be a character. Then $\pi^{\chi}(g) := \chi^{1/d}(g)\pi(g)$
is again a representation, and
\[\mathrm{det}(\pi^{\chi}(g)) = \chi(g)\mathrm{det}(\pi(g)).\]
The following lemma allows one to split a representation $(\pi,V)$ into a subrepresentation
$W \subseteq V$ and a quotient $V/W$, and simultaneously twist $\pi_W$ and
$\pi_{V/W}$ by a character $\chi$ in opposite directions.


\begin{Lemma}[Shifting characters]\label{Lemma:shift}
Let $W \subseteq V$ be a (real) subrepresentation, and let $\cV_{W}$ and $\cV_{V/W}$ be open neighbourhood bases in $W$ and $V/W$ consisting of balanced sets. Let $F \subseteq G$ be a relatively compact, symmetric subset.
Then
\begin{equation}
 \delta_{F}^{\pi} \geq \delta_{F}^{\pi_{W}^{\chi^{-1}}}(\cV_{W})\cdot
 \delta_{F}^{\pi_{V/W}^{\chi}}(\cV_{V/W})
\end{equation}
for any continuous character $\chi \colon G \rightarrow \R^{+}$.
\end{Lemma}
\begin{proof}
Choose an inner product on $V$, and identify $V/W$ with $W^{\perp}\subseteq V$.
Let $d_1 = \mathrm{dim}(W)$ and $d_2 = \mathrm{dim}(V/W)$.
Define $\rho := \sup_{g\in F}\|\pi(g)\|$ and $D := \sup_{g\in F}|\chi(g)|$.

Let $U\in \cV_{U}$, let $V \in \cV_{V/W}$, and suppose that both are contained in a ball with radius 1.
Let
$(\alpha u, \beta v)$ be an element of $(U_{r})^{F} \times_{R_1,R_2,\varepsilon} V^F$ for
\begin{equation}
r = \rho R_2^{1 + \frac{d_2}{d_1}}(D/\varepsilon)^{1/d_1}.
\end{equation}
Then
\[
\pi(g)\begin{pmatrix}\alpha u\\ \beta v\end{pmatrix} =
\begin{pmatrix}
\alpha' \pi^{\chi^{-1}}_{W}(g)u + \beta P_{W}\pi(g)v\\
\beta'  \pi^{\chi}_{V/W}(g) v\end{pmatrix}
\]
with $\alpha' = \alpha \chi^{1/d_1}(g)$ and  $\beta' = \beta\chi(g)^{-1/d_2}$,
and $P_W$ the orthogonal projection onto $W$.
If $\alpha \neq 0$,
this can be written as $(\alpha'u', \beta'v')$ for 
\[v' = \pi^{\chi}_{V/W}(g) v, \quad \text{and} \quad
u' = \pi_{W}^{\chi^{-1}}(g)u + \frac{\beta}{\alpha'}P_W\pi(g)v.
\] 

Note that since $v \in V^{F}$, we have $v' \in V$.
To see that $u' \in U$, note that since $u \in (U_{r})^F$ we have $\pi_{W}^{\chi^{-1}}(g)u \in U_r$.
Further, $\|P_W \pi(g)w\| \leq \rho$. 
Since $U$ is balanced, $U_r$ is balanced as well. 
By Lemma~\ref{lem:balanceboundary}, we may assume that 
$|\alpha^{d_1}\beta^{d_2}| = \varepsilon$. Then
\[\frac{|\beta|}{|\alpha|} \leq (1/\varepsilon)^{1/d_1}R_2^{1 + \frac{d_2}{d_1}},\] and
$\frac{|\beta|}{|\alpha'|} \leq (D/\varepsilon)^{1/d_1}R_2^{1 + \frac{d_2}{d_1}}$. 
Putting this together, we find $\|\frac{\beta}{\alpha'}P_W\pi(g)v\| \leq r$
for $r$ as above, and we conclude that $u'\in U$.

Since $|\alpha'{}^{d_1}\beta'{}^{d_2}| = |\alpha^{d_1}\beta^{d_2}| = \varepsilon$ and $|\alpha'|\leq D^{1/d_1}R_1 = : R'_1$, and
$|\beta'| \leq D^{1/d_2}R_2 = : R'_2$, we conclude that
\[
 \pi(g)\Big((U_{r})^{F} \times_{R_1,R_2,\varepsilon} V^F\Big) \subseteq U \times_{R_1',R_2',\varepsilon}V
\]
for all $g\in F$, and hence, since $F$ is symmetric,
\[
	(U \times_{R_1',R_2',\varepsilon}V)^{F} \supseteq (U_{r})^{F} \times_{R_1,R_2,\varepsilon} V^F.
\]
By Lemma~\ref{lem:logaritmetruc}, it follows that
\begin{eqnarray}
 \delta_{F}^{V}(U \times_{R_1',R_2',\varepsilon}V) &\geq& \frac{\Lambda\Big((U_{r})^{F} \times_{R_1,R_2,\varepsilon} V^F\Big)}{\Lambda\Big(U \times_{R_1',R_2',\varepsilon}V\Big)} \label{eq:bakbeest}\\
 &=& \frac{\Lambda\big((U_{r})^{F}\big)}{\Lambda(U)} \cdot \frac{\Lambda(V^{F})}{\Lambda(V)}\cdot
 \frac{
	\varepsilon\left(1 + \log\Big(\frac{R_1^{d_1}R_2^{d_2}}{\varepsilon}\Big)\right)
	}
	{
	\varepsilon\left(1 + \log\Big(\frac{R'_{1}{}^{d_1}R'_2{}^{d_2}}{\varepsilon}\Big)\right)
	}.\nonumber
\end{eqnarray}

We choose $R_1(\varepsilon)$ and $R_2(\varepsilon)$ such that
$R_1(\varepsilon)\downarrow 0$ and $R_2(\varepsilon)\downarrow 0$ as $\varepsilon\downarrow 0$,
for example by setting $R_1(\varepsilon) = \varepsilon^a$ and $R_2(\varepsilon) = \varepsilon^b$
with $a,b>0$.
Further, we can achieve $r(\varepsilon) \downarrow 0$ for
\[r(\varepsilon) = \rho D^{1/d_1}\frac{R_2(\varepsilon)^{1 + \frac{d_2}{d_1}}}{\varepsilon^{1/d_1}},\]
by choosing $b >1/(d_1 + d_2)$.
The first factor in \eqref{eq:bakbeest} then converges to the constant $\delta_{F}^{\pi^{\chi^{-1}}_{W}}(U) = \Lambda(U^{F})/\Lambda(U)$.
The second factor is $\delta_{F}^{\pi^{\chi}_{V/W}}(V)$.
For the third factor,
we can achieve
\[\frac{R_1^{d_1}(\varepsilon)R_2(\varepsilon)^{d_2}}{\varepsilon} \rightarrow \infty\] for $\varepsilon \downarrow 0$ by choosing $a>0$ such that $d_1a + d_2 b < 1$.
The third factor in~\eqref{eq:bakbeest} then converges to $1$, because $R'_1{}^{d_1}R_2'{}^{d_2}$ differs from $R_1{}^{d_1}R_2{}^{d_2}$ by a multiplicative constant $D^2$.
Note that the inequalities $a>0$, $d_1a + d_2 b < 1$ and $b>\frac{1}{d_1 + d_2}$ do admit solutions.
\end{proof}


\section{Connected Lie groups}

If $G$ is a connected Lie group and $(\pi, V)$ is a finite dimensional real $G$-re\-pre\-sen\-ta\-tion, then the character
$\chi(g) = \mathrm{det}(\pi(g))$ takes values in $\R^{>0}$. 
\begin{Definition}[Rescaled representations]
We denote by 
 $
 \overline{\pi}(g) := \mathrm{det}\big(\pi(g)\big)^{-\frac{1}{d}}\pi(g)
$
 the representation on $V$ that is rescaled to have determinant 1.
 \end{Definition}
 
 \begin{Definition}
 We denote by $\overline{\delta}{}_{F}^{V}$ the supremum of $\delta_{F}^{(\overline{\pi},V)}(\cV)$ over all \emph{balanced} open neighbourhood bases
 of $0\in V$.
 \end{Definition}

 Note that if $(\pi,V)$ has determinant 1, then $\delta_{F}^{V}\geq \overline{\delta}{}^{V}_{F}$.
 From Lemma~\ref{Lemma:shift}, we immediately obtain the following splitting theorem:
 \begin{Lemma}[Rescaled splitting]\label{Lemma:RescaledSplitting}
 Let $G$ be a connected Lie group, and $W \subseteq V$ a subrepresentation of the (real)
 representation $(\pi, V)$. Then
\begin{equation}
 \overline{\delta}{}_{F}^{V} \geq \overline{\delta}{}_{F}^{W} \cdot
 \overline{\delta}{}_{F}^{V/W}.
\end{equation}
\end{Lemma}

\subsection{Connected solvable Lie groups}

Recall that a connected Lie group $R$ is solvable if and only if its Lie algebra $\fr$ is solvable \cite[Theorem~3.18.8]{Varadarajan}.
From the above rescaled splitting lemma, it is not hard to see that $\overline{\delta}{}_{F}^{V} = 1$ for connected solvable Lie groups.

\begin{Theorem}
Let $R$ be a connected Lie group with solvable Lie algebra $\fr$, and let
$F \subseteq R$ be a relatively compact, symmetric subset.
Then $\overline{\delta}{}_{F}^{V} = 1$ for any real representation $(\pi,V)$.
\end{Theorem}
\begin{proof}
We proceed by induction on $d = \mathrm{dim}_{\R}(V)$.
For $d=1$, the statement holds because $\overline{\pi}$ is the trivial representation.
Suppose the statement holds for all representations of dimension less than $d$.
By Lie's Theorem, the complexification $V_{\C} = V \otimes_{\R}\C$ admits a nonzero vector $v_{\lambda}$ 
and a weight $\lambda \colon \fr \rightarrow \C$
such that $d\pi(x)v_{\lambda} = \lambda(x) v_{\lambda}$ for all $x\in \fr$.
Since $R$ is connected and since the infinitesimal action of $\fr$ on $\C\mathrm{P}(V)$ has $[v_{\lambda}] \in \C\mathrm{P}(V)$
as a fixed point, the action of $R$ on $\C\mathrm{P}(V)$ fixes the ray $[v_{\lambda}] = \{z v_{\lambda}\,;\, z \in \C^{\times}\}$.
The
homomorphism $\lambda$ therefore integrates to a character $e_{\lambda} \colon R \rightarrow \C^{\times}$,
and $\pi(g)v_{\lambda} = e_{\lambda}(g)v_{\lambda}$ for all $g\in R$. 
Let $U_{\lambda} \subseteq V$ be the real subrepresentation spanned by the real vectors $\frac{1}{2}(v_{\lambda} + \overline{v}_{\lambda})$ and 
$\frac{1}{2i}(v_{\lambda} - \overline{v}_{\lambda})$, and let $\chi(g) := |e_{\lambda}(g)|$.
If $v_{\lambda}$ is real, then $U_{\lambda}$ is 1-dimensional, and the representation $(e_{\lambda}, V_{\lambda})$
becomes trivial when rescaled to have determinant $1$.
If $v_{\lambda}$ is not real, then $U_{\lambda}$ is 2-dimensional, and the real representation $\overline{e}_{\lambda}(g) = \chi(g)^{-1}e_{\lambda}(g)$ (rescaled to have determinant 1) acts by rotations.
Either way, $U_{\lambda}$ admits an invariant neighbourhood basis for the rescaled representation, so 
\[\overline{\delta}{}_{F}^{U_{\lambda}} =1.\]
Since the quotient representation $(\pi_{V/W}, V/U_{\lambda})$ is of dimension less than $d$, it too has $\overline{\delta}{}_{F}^{V/U_{\lambda}} = 1$. 
So $\overline{\delta}{}_{F}^{V} \geq 1$ by  Lemma~\ref{Lemma:RescaledSplitting}.
\end{proof}

Since a connected Lie group is unimodular if and only if the adjoint representation has determinant one, 
we have  $1 \geq \delta^{\fr}_{F} \geq \overline{\delta}{}^{\fr}_{F}$, and
the following result is an immediate consequence.
\begin{Corollary}\label{Thm:Solvable}
If $R$ is a connected, unimodular Lie group with solvable Lie algebra $\fr$, then $\delta^{F}_{\fr} = 1$ for any relatively compact, symmetric subset $F \subseteq R$.
\end{Corollary}


From Theorem~\ref{Thm:ncDeLeeuw}, it then follows that 
\begin{equation}\label{eq:boundSolvable}
\|T_{m|_{\Gamma}} \colon L_p(\widehat{\Gamma}) \rightarrow L_p(\widehat{\Gamma})\| \leq 
\|T_{m} \colon L_p(\widehat{R}) \rightarrow L_p(\widehat{R})\|
\end{equation}
for any discrete subgroup $\Gamma$ of a connected, solvable Lie group $R$.

This is in agreement with results obtained earlier in \cite{CPPR15}; since every solvable Lie group is amenable \cite{Day57}, and since every 
discrete subgroup $\Gamma$ of an amenable group is amenable, 
it follows from \cite[Theorem~8.7]{CPPR15} that $R$ has small almost-invariant neighbourhoods with respect to $\Gamma$,
and \eqref{eq:boundSolvable} follows from \cite[Theorem~A]{CPPR15}.

\subsection{Connected Lie groups}\label{sec:ConnectedLieGroups}

We end with some remarks on $\delta^{\fg}_{F}$ for a connected unimodular Lie group $G$ with Lie algebra $\fg$.
Let 
$\fr \subseteq \fg$ be the radical of $\fg$, i.e.\ the maximal solvable ideal. 
Then $\fr$ integrates to a closed, connected, normal Lie subgroup $R \subseteq G$, 
and the semisimple quotient $S := G/R$ has Lie algebra $\fs := \fg/\fr$ \cite[Thm.~3.18.13]{Varadarajan}.
Let $F \subseteq G$ be a relatively compact, symmetric subset.
Using the Reduction Lemma~\ref{Lemma:BenjaminsTruc} on the $G$-subrepresentation $\fr \subseteq \fg$, we obtain a bound 
\begin{equation}\delta_{F}^{\fg} \geq a_{\mathrm{ss}} a_{\mathrm{rad}} \end{equation} 
with two types of contributions: the `semisimple' contribution $a_{\mathrm{ss}} = \delta^{\fg/\fr}_{F}$ and the 
contribution $a_{\mathrm{rad}} = \delta_{F}^{\fr}$ from the radical.

\subsubsection{The semisimple contribution $a_{\mathrm{ss}}$ from $\fs = \fg/\fr$.}
The factor $a_{\mathrm{ss}} := \delta^{\fg/\fr}_{F}$ depends only on the semisimple quotient $S$.
Indeed, since the adjoint action of $\fr$ on $\fs = \fg/\fr$ is trivial, the action of $G$ on $\fs$ is trivial on the connected subgroup $R$, and
the representation of $G$ on $\fg/\fr$ factors through the adjoint representation of $S$.
It follows that $a_{\mathrm{ss}}$ is equal to $\delta^{\fs}_{[F]}$ for $[F] \subseteq S$.
Let $B$ be the Killing form on $\fs$, and $\theta$ a Cartan involution.
By \cite[Theorem~8.1]{CJKM23}, we have 
\begin{equation}
a_{\mathrm{ss}} \geq \rho_{S}^{-d/2},
\end{equation}
where $\rho_{S} = \sup_{s\in [F]}\|\mathrm{Ad}_{s}\|$ is the maximal norm of $s \in [F]\subseteq S$ in the adjoint representation, 
with respect to the inner product $B_{\theta}(x,y) = -B(x, \theta(y))$. The exponent $d$ is the maximal dimension of a nilpotent orbit in $\fs$.
Note that $d$ is an even integer, 
since the adjoint orbit in a semisimple Lie algebra is a symplectic manifold. 
In terms of the Lie algebra $\fs$, it can be expressed as 
$d = \mathrm{dim}_{\R}(\fs) - \mathrm{min}_{x\in \cN} \mathrm{dim}_{\R}(\fs_{x})$, where the second term is the minimal dimension 
of the stabilizer $\fs_{x}$, where $x$ runs over the nilpotent cone $\cN \subseteq \fs$ \cite[Remark~8.2]{CJKM23}.

\subsubsection{The contribution $a_{\mathrm{rad}}$ from the radical.}

Since $G$ is unimodular, the adjoint representation $(\Ad_{G}, \fg)$ has determinant~1. 
 As the semisimple quotient $S = G/R$ is unimodular as well, 
 the adjoint representation $(\Ad_{S}, \fs)$ of $S$ also has determinant 1.
 It follows that the restriction $(\pi, \fr)$ of the adjoint $G$-representation to the radical $\fr\subseteq \fg$ has 
 determinant 1 as well, so that $a_{\mathrm{rad}} = \delta^{\fr}_{F} \geq \overline{\delta}{}_{F}^{\fr}$.

Recall from \S\ref{sec:LieThm} that the real $G$-representation $V^k(\fr) = \fr^k/\fr^{k+1}$ decomposes as 
\begin{equation}
V^k(\fr) = \bigoplus_{\{\lambda, \overline{\lambda}\} \subseteq  \Phi^k(\fr_{\C})} U_{\lambda}^k,
\end{equation}
where $U^{k}_{\lambda}$ is a real form of the complex representation $V^{k}_{\lambda}$ if $\lambda$ is real, and 
$U^k_{\lambda} \simeq V^{k}_{\lambda}$ as a real representation otherwise.
From Lemma~\ref{Lemma:RescaledSplitting}, we immediately obtain the following result.
\begin{Lemma}\label{Lem:BenjaminsTrucHerhaald}
Let $F \subseteq G$ be a relatively compact, symmetric subset. Then
\begin{equation}\label{eq:decompradpart}
\delta^{\fr}_{F} \geq \prod_{k=0}^{r-1} \prod_{\{\lambda, \overline{\lambda}\} \subseteq \Phi^k(\fr_{\C})} \overline{\delta}^{U_{\lambda}^{k}}_{F}.
\end{equation}
\end{Lemma}
\begin{proof}
Since $\fr^k/\fr^{k+1} = \bigoplus_{\{\lambda, \overline{\lambda}\} \subseteq  \Phi^k(\fr_{\C})} U_{\lambda}^k$, Lemma~\ref{Lemma:RescaledSplitting} 
yields
\begin{equation}\label{eq:onerow}
\overline{\delta}^{\fr^k/\fr^{k+1}}_{F}\geq  \prod_{\{\lambda, \overline{\lambda}\} \subseteq \Phi^k(\fr_{\C})} \overline{\delta}^{U_{\lambda}^{k}}_{F}.
\end{equation}
The result then follows by applying Lemma~\ref{Lemma:RescaledSplitting} to the subrepresentations $\fr^k/\fr^{k+1}$ of $\fr/\fr^{k+1}$ repeatedly 
for $k = 2, \ldots, r-1$.
\end{proof}

If $\lambda$ is real, then the maximal solvable subgroup $R\subseteq G$ acts on the real subspace $U_{\lambda}^k = V^k_{\lambda} \cap V^k(\fr)$
by the real character 
$e_{\lambda} \colon R \rightarrow \R^{>0}$.  The rescaled representation $\overline{\pi}{}^{k}_{\lambda}$ of $G$ is therefore trivial on $R \subseteq G$,
and factors through a real representation of the semisimple quotient $S = G/R$.

If $\lambda$ is not real, then $R$ acts on $U^{k}_{\lambda} = V_{\lambda}^{k}$ by the (complex) character $e_{\lambda} \colon R \rightarrow \C^{\times}$, 
and the restriction of $\overline{\pi}{}_{\lambda}^{k}$ to $R\subseteq G$
factors through a homomorphism $\chi_{\lambda} \colon R \rightarrow \mathrm{U}(1)$. In fact, $\chi_{\lambda}$ is trivial 
on the commutator subgroup $[R,R]$ because $\lambda$ vanishes on $[\fr,\fr]$, so it factors through a
compact quotient of the abelianization $R/[R,R]$ of $R$.
Since $G$ is connected, it acts trivially on the weight $\lambda$, and hence on the character $\chi_{\lambda}$. 
It follows that $\chi_{\lambda}(grg^{-1}) = \chi_{\lambda}(r)$ for all $r\in R$ and $g\in G$.
The representation $\overline{\pi}_{\lambda}$ of $G$ 
on $U_{\lambda}^{k}$ therefore factors through the central extension 
\begin{equation}\label{eq:centralextension}
S_{\lambda}^{\sharp} := G \times_{\chi_{\lambda}} \mathrm{U}(1),
\end{equation}
the quotient of $G \times \mathrm{U}(1)$ by the closed, normal subgroup $\{(r^{-1}, \chi_{\lambda}(r))\,;\, r \in R\}$ isomorphic to $R$. Note that 
\begin{equation}\label{eq:centralextension2}
1 \rightarrow \mathrm{U}(1) \rightarrow S^{\sharp}_{\lambda} \rightarrow S \rightarrow 1
\end{equation}
is a central extension of the semisimple Lie group $S$ by the abelian Lie group $\mathrm{U}(1)$, and hence a real reductive Lie group.
The relevant subset of $S^{\sharp}_{\lambda}$ is the image $[F] = \{[g, 1]\,;\, g \in F\}$ of $F \times \{1\}$ under the quotient map, and the 
factors $\overline{\delta}{}^{U^{k}_{\lambda}}_{F}$ in \eqref{eq:decompradpart} come from the representation of $S^{\sharp}_{\lambda}$ on $U_{\lambda}^{k}$.

\begin{Remark}
In fact, the central extension \eqref{eq:centralextension2} splits at the Lie algebra level by the second Whitehead Lemma.
It follows that $S^{\sharp}_{\lambda}$ can be expressed in terms of the simply connected cover $\widetilde{S}$ 
and a character $\phi_{\lambda} \colon \pi_1(S) \rightarrow \mathrm{U}(1)$ 
by \[S^{\sharp}_{\lambda} = \widetilde{S} \times_{\phi_{\lambda}} \mathrm{U}(1),\] the quotient of 
$\widetilde{S} \times \mathrm{U}(1)$ by the normal subgroup $\{([\gamma]^{-1}, \phi_{\lambda}([\gamma]))\,;\, [\gamma] \in \pi_1(S)\}$.
The inclusion $\mathfrak{s} \hookrightarrow \mathfrak{g}$ of Lie algebras integrates to a Lie group homomorphism
$\phi \colon \widetilde{S} \rightarrow G$, and the restriction of $\phi$ to $\pi_1(S)$ takes values in $R$.
In terms of $\chi_{\lambda}$ and $\phi$, the character $\phi_{\lambda}$ is given by $\phi_{\lambda} = \chi_{\lambda} \circ \phi|_{\pi_1(S)}$.
\end{Remark}

\begin{Remark}
If we consider $U^k_{\lambda} = V^k_{\lambda}$ as a \emph{complex} representation for $\lambda \neq \overline{\lambda}$, and consider only neighbourhood bases 
that consist of \emph{complex} balanced open subsets (so $\alpha U \subseteq U$ for all \emph{complex} numbers $\alpha$ with $|\alpha| \leq 1$), 
then the $\mathrm{U}(1)$-action is immaterial. So at the cost of restricting attention from $\R$-balanced neighbourhoods 
to $\C$-balanced neighbourhoods, the real reductive Lie groups $S^{\sharp}_{\lambda}$ can be replaced by the
connected semisimple Lie group $\widetilde{S}$.
\end{Remark}


Note that $\mathrm{dim}_{\R}(U_{\lambda}^{k}) = \mathrm{dim}_{\C}(V_{\lambda}^{k})$ if $\lambda$ is real, and 
$\mathrm{dim}_{\R}(U_{\lambda}^{k}) = \mathrm{dim}_{\C}(V_{\lambda}^{k}) + \mathrm{dim}_{\C}(V_{\overline{\lambda}}^{k})$
if $\lambda$ is not real. Applying the trivial bound \eqref{eq:TrivialBound} to each of the factors $\delta^{U^k_{\lambda}}_{[F]}$
in \eqref{eq:decompradpart} therefore yields the explicit bound
\begin{equation}\label{eq:grofrad}
a_{\mathrm{rad}} \geq \prod_{\lambda \in \Phi(\fr_{\C})} (\rho^{[F]}_{\lambda})^{-\dim_{\C}(V^{k}_{\lambda})},
\end{equation}
where $\rho^{[F]}_{\lambda}$ is the maximal value of $\|\overline{\pi}_{\lambda}(g) \colon V^{k}_{\lambda} \rightarrow V^{k}_{\lambda}\|$ for $[g] \in [F] \subseteq S^{\sharp}_{\lambda}$.
If $\fg$ is solvable, then $G = R$, $S$ is trivial, and $S^{\sharp}_{\lambda}$ is a compact quotient of $R/[R,R]$, acting by operators with norm 1. 
The bound \eqref{eq:grofrad} then reduces to Theorem~\ref{Thm:Solvable}, which is of course sharp.

In general, however, the trivial bound on $\delta^{U^k_{\lambda}}_{[F]}$ will not be sharp. 
A case in point is the tangent group $TS = A \rtimes S$ of a semisimple Lie group $S$, the semidirect product of $S$
with its Lie algebra $A = \fs$, considered as an abelian Lie group with the adjoint action of $S$.
The trivial bound then yields $\delta^{\fa \rtimes \fs}_{F} \geq \rho^{-d/2}\rho^{-\mathrm{dim}_{\R}(\fs)}$ for 
$\rho = \sup_{s\in F}\|\mathrm{Ad}_s\|$. But since $\fr = \fs$ as an $S$-representation,
\cite[Theorem~8.1]{CJKM23}
yields the superior bound $\delta^{\fa \rtimes \fs}_{F} \geq \rho^{-d/2}\rho^{-d/2}$.

\subsection{Open problems and further directions of research}

This leads one to consider the following problem:
\begin{Problem}\label{OpenProblem}
Determine nontrivial bounds on $\overline{\delta}{}_{F}^{V}$, where $F$ is a compact subset of a connected real
reductive Lie group $H$, and $(\pi, V)$ is a finite-dimensional, irreducible, real representation of $H$ with $\mathrm{det}(\pi(h)) = 1$
for all $h\in H$.
\end{Problem}
By `nontrivial', we mean superior to the trivial bound from Proposition~\ref{Prop:trivialbound}.
Except for the case where $V$ is the adjoint representation \cite[Theorem~8.1]{CJKM23}, this is an open problem to the best of our knowledge.

By the above considerations, a solution to Problem~\ref{OpenProblem} would imply nontrivial bounds on $\delta^{\fg}_{F}$ for a general connected unimodular Lie group $G$.
Indeed, if one can find nontrivial bounds in the case where $H = S^{\sharp}_{\lambda}$ and $V$ is an irreducible real subrepresentation of $U^{k}_{\lambda}$, 
then repeated application of the Splitting Lemma~\ref{Lemma:RescaledSplitting} will result in nontrivial bounds for 
$\overline{\delta}{}^{U^{k}_{\lambda}}_{[F]}$, and hence for $\delta^{\fg}_{F}$.

Conversely, suppose that $H$ is a connected real reductive Lie group and $(\pi, V)$ is an irreducible, finite-dimensional, 
real representation of $H$ with determinant one. Then for the connected unimodular Lie group $G = V \rtimes H$,
it seems unlikely that one would be able to arrive at nontrivial bounds on $\delta^{\fg}_{F}$ for $G$ without first considering $\overline{\delta}{}_{[F]}^{V}$
for $H$.

Rather than considering neighbourhood bases of balls $B_{r}(0) \subseteq V$, which lead to the trivial bound \eqref{eq:TrivialBound}, it would be 
a natural direction of research
(following \cite[\S 9]{CJKM23}) to consider neighbourhood bases $\cV$ consisting of sets $U_{r,R} \subseteq V$ of the form 
\[
U_{r,R} :=
\Big(\pi(H) B_{r}(0)\Big) \cap B_{R}(0)
\]
for $0 < r \ll R < 1$. Note that $\bigcap_{r > 0}\bigcup_{R>0} U_{r,R}$ is the union of the $H$-orbits in $V$ whose closure contains the origin.
Although this would probably require detailed information about the orbit structure of the $H$-action on $V$, we believe that
determining $\delta^{V}_{F}(\cV)$ for such neighbourhood bases has the potential to yield nontrivial bounds on $a_{\mathrm{rad}}$.
This has been substantiated in the special case of the adjoint representation, where this is precisely the Ansatz used in 
{\cite[\S 9]{CJKM23}}.


\begin{thebibliography}{CJKM23}

\bibitem[Cas13]{Caspers13}
    M. Caspers,
    \emph{The $L^p$-Fourier transform on locally compact quantum groups},
    J. Operator Theory {\bf 69} (2013), 161--193.


\bibitem[CPPR15]{CPPR15}
M.\ Caspers, J.\ Parcet, M.\ Perrin, and \'{E}.\ Ricard, \emph{Noncommutative De Leeuw Theorems}, Forum Math., Sigma, Vol.~3, e21 (2015), 
59p.


\bibitem[CJKM23]{CJKM23}
M.\ Caspers, B.\ Janssens, A.\ Krishnaswamy-Usha, L.\ Miaskiwskyi, \emph{Local and multilinear noncommutative de Leeuw theorems}, 
Math. Ann. {\bf 388} (2024), 4251--4305

\bibitem[Coo10]{Cooney10}
   T. Cooney,
   \emph{A Hausdorff-Young inequality for locally compact quantum groups},
   Internat. J. Math. {\bf 21} (2010), 
   1619--1632.
   
\bibitem[Day57]{Day57}
Mahlon M.\ Day, \emph{Amenable semigroups}, Illinois \ J. Math. {\bf 1} (1957), 509--544.

\bibitem[GJP17]{JungeGonzalezParcet}
  A. Gonz\'alez-P\'erez, M. Junge, J. Parcet,
  \emph{Smooth Fourier multipliers in group algebras via Sobolev dimension},
  Ann. Sci. \'Ec. Norm. Sup\'er. (4) {\bf 50} (2017), 
  879--925.


\bibitem[Gra08]{GrafakosBook1}
  L. Grafakos,
  \emph{Classical Fourier analysis},
  Second edition. Graduate Texts in Mathematics, 249. Springer, New York, 2008.

\bibitem[Gra09]{GrafakosBook2}
  L. Grafakos,
  \emph{Modern Fourier analysis},
  Second edition. Graduate Texts in Mathematics, 250. Springer, New York, 2009.



\bibitem[JMP14]{JMPGafa}
   M.\ Junge, T.\ Mei, J.\ Parcet,
   \emph{Smooth Fourier multipliers on group von Neumann algebras},
   Geom. Funct. Anal. {\bf 24} (2014), 
   1913--1980.

  
  \bibitem[JMP18]{JMPJEMS}
   M.\ Junge, T.\  Mei, J.\ Parcet,
   \emph{Noncommutative Riesz transforms -- dimension free bounds and Fourier multipliers},
    J. Eur. Math. Soc. (JEMS) {\bf 20 } (2018), 
    529--595.
    
    
\bibitem[PRS22]{PRS}
   J.\ Parcet, \'E.\ Ricard,  M.\ de la Salle,
   \emph{Fourier multipliers in $\SL_n(\mathbb{R})$},
    Duke Math. J. {\bf 171} (2022), 
    1235--1297.

\bibitem[MRX22]{MeiRicardXu}
   T. Mei, \'E. Ricard, Q. Xu,
\emph{A Mikhlin multiplier theory for free groups and amalgamated free products of von Neumann algebras},
Adv. Math. {\bf 403} (2022), 108394.

\bibitem[MeRi17]{MeiRicard}
  T.~Mei, \'E.~Ricard,
  \emph{Free Hilbert transforms},
  Duke Math. J. {\bf 166} (2017), 
  2153--2182.


\bibitem[Kun58]{Kunze58}
 R. Kunze,
 \emph{$L_p$ Fourier transforms on locally compact unimodular groups},
 Trans. Amer. Math. Soc. {\bf 89} (1958), 519--540.

\bibitem[Lee65]{DeLeeuw}
   K. de Leeuw,
   \emph{On $L_p$ multipliers},
   Ann. of Math. (2) {\bf 81} (1965), 364--379.

\bibitem[Ste70]{Stein}
  E. Stein,
  \emph{Singular integrals and differentiability properties of functions},
  Princeton Mathematical Series, No. 30 Princeton University Press, Princeton, N.J. 1970 xiv+290 pp.



\bibitem[V84]{Varadarajan}
V.S.~Varadarajan, Lie Groups, Lie Algebras, and Their Representations, (reprint of the 1974 Prentice-Hall edition), Grad. Texts in Math. {\bf 102}, Springer, New York, 1984. 

\end{thebibliography}
\end{document}